\documentclass{amsart} 

\usepackage{graphicx}
\usepackage{amsmath}
\usepackage{amssymb}
\usepackage{ascmac}
\usepackage{amsthm}

\newtheorem{thm}{Theorem}[section]
\newtheorem{lem}{Lemma}[section]
\newtheorem{prop}{Proposition}[section]
\newtheorem{cor}{Corollary}[section]
\newtheorem{defn}{Definition}[section]
\newtheorem{rem}{Remark}[section]

\title[A construction of special Lagrangian submanifolds]{A construction of special Lagrangian submanifolds by generalized perpendicular symmetries}

\author{Akifumi Ochiai}
\date{}

\subjclass{53C38}
\keywords{special Lagrangian submanifold, Calabi-Yau manifold, minimal submanifold, moment map}
\email{akfm.oc@gmail.com, a-ochiai@tmu.ac.jp}

\begin{document}
\maketitle

\begin{abstract}
We show a method to construct a special Lagrangian submanifold $L^\prime$ from a given special Lagrangian submanifold $L$ in a Calabi-Yau manifold with the use of generalized perpendicular symmetries. We use moment maps of the actions of Lie groups, which are not necessarily abelian. By our method, we construct some non-trivial examples in non-flat Calabi-Yau manifolds $\mathrm{T}^\ast S^n$ which equipped with the Stenzel metrics.
\end{abstract}

\section{Introduction}\label{Section.A}

In 1982, Harvey and Lawson introduced a special class of submanifolds, namely calibrated submanifolds in their paper \cite{Harvey and Lawson}. Calibrated submanifolds has a strong property that they realize volume minimizing submanifolds in the homological class. Particularly, in Calabi-Yau manifolds calibrated submanifolds which have half-dimensions are defined and they are called special Lagrangian submanifolds. Because special Lagrangian submanifolds play an important role for understanding mirror symmetries and the SYZ-conjecture, many mathematicians pay attention to their constructions and singularities.   

Let us review the history of constructions of special Lagrangian submanifolds, regarding their ambient spaces and methods of constructions. At first $\mathbb{C}^n$ was chosen for an ambient space and in there various examples and methods of constructing special Lagrangian submanifolds were given by Joyce in a series of his papers \cite{Joyce2001}--\cite{Joyce2008}. On the other hand, Stenzel gave examples of non-flat Calabi-Yau structures on the conormal bundles over compact rank one symmetric spaces. Next special Lagrangian submanifolds are constructed in those spaces (first in $\mathrm{T}^\ast S^n$, and recently in $\mathrm{T}^\ast \mathbb{C}P^n$).

One of the useful method of constructing special Lagrangian submanifolds is called moment map techniques which were introduced by Joyce in \cite{Joyce2002}. This method needs large symmetries, and by using these symmetries we can reduce PDEs for being special Lagrangian submanifolds to ODEs on the orbit spaces. Using this method, Joyce constructed special Lagrangian submanifolds in $\mathbb{C}^n(\cong\mathrm{T}^\ast\mathbb{R}^n)$ invariant under a subgroup of $SU(n)$. With this method special Lagrangian submanifolds were also studied in $\mathrm{T}^\ast S^n$ by Anciaux \cite{Anciaux}, Ionel and Min-Oo \cite{Ionel and Min-Oo}, Hashimoto and Sakai \cite{Hashimoto and Sakai}, Hashimoto and Mashimo \cite{Hashimoto and Mashimo}, and in $\mathrm{T}^\ast\mathbb{C}P^n$ by Arai and Baba \cite{Arai and Baba}. All of these examples were cohomogeneity one.

Another method was introduced by Harvey and Lawson \cite{Harvey and Lawson} which is called bundle techniques. With the use of this method, Karigiannis and Min-Oo \cite{Karigiannis and Min-Oo} constructed special Lagrangian submanifolds in $\mathrm{T}^\ast S^n$, and Ionel and Ivey \cite{Ionel and Ivey} in $\mathrm{T}^\ast\mathbb{C}P^n$.

Aside from these two typical methods, Joyce \cite{Joyce2002} showed a way to construct a special Lagrangian submanifold $L^\prime$ in $\mathbb{C}^n$ from another given special Lagrangian submanifold $L$ by using actions of an abelian group which acts perpendicularly to $L$. This method has advantage that we need not deal with the PDE for $L^\prime$ to be a special Lagrangian submanifold (it is ``already achieved'' by the given special Lagrangian submanifold $L$), and that large symmetries are not necessarily needed.

In this paper we generalize  this Joyce's result above using ``perpendicular symmetries'' in three points. Firstly we generalize ambient spaces to general Calabi-Yau manifolds from $\mathbb{C}^n$. Secondly we do not assume the commutativity of Lie groups. Thirdly we generalize the condition that the group acts perpendicularly to a given special Lagrangian submanifold. By this method we also construct non-trivial examples of special Lagrangian submanifolds in Calabi-Yau manifolds $\mathrm{T}^\ast S^n$ equipped with the Stenzel metrics. 

The method to construct special Lagrangian submanifolds in this paper is summarized as follows: Let $(M,I,\omega,\Omega)$ be a connected Calabi-Yau manifold and $H$ a connected Lie group which acts on $M$ preserving $I$. Let $\mathfrak{h}$ and $\mathfrak{h}^\ast$ be the Lie algebra of $H$ and its dual respectively. Assume the $H$-action is Hamiltonian, i.e. $(M, \omega, H)$ has a moment map $\mu:M\to\mathfrak{h}^\ast$. Let $L$ be a special Lagrangian submanifold of $(M, I, \omega, \Omega)$ and $Z(\mathfrak{h}^\ast)$ the center of $\mathfrak{h}^\ast$. Suppose that for $c\in Z(\mathfrak{h}^\ast)$, $V_c$ is a submanifold of $M$ which satisfies $V_c\subset \mu^{-1}(c)\cap L$ and $\dim H + \dim V_c =\frac{1}{2}\dim M$.  Assume that the actions of $H$ are ``(generalized) perpendicular actions'' for $V_c$ (not necessarily for whole of $L$). Then $H\cdot V_c$ is a special Lagrangian submanifold.

Konno \cite{Konno} showed in general Calabi-Yau manifolds a method of constructing Lagrangian mean curvature flows by using perpendicular actions of abelian groups for given special Lagrangian submanifolds, and constructed some examples. This paper is inspired from the study by Konno. 

\section{Preliminaries}\label{Section.B}

In this section, we review some fundamental facts about Calabi-Yau manifolds, their special Lagrangian submanifolds, group actions, and moment maps.

\subsection{Special Lagrangian submanifolds}\label{Section.B.1}

We begin with the definition of Lagrangian submanifolds in symplectic manifolds.

Let $(M,\omega)$ be a symplectic manifold. A submanifold $L$ of $(M,\omega)$ is \textit{isotropic} if $\omega|_L\equiv0$. If an isotropic submanifold $L$ is of half-dimension of $\dim M$, it is called a \textit{Lagrangian submanifold}.

Next we see the definition of special Lagrangian submanifolds. It is a particular submanifold of a Calabi-Yau manifold which is defined as follows:
\begin{defn}
A \textit{Calabi-Yau manifold} is a quadruple $(M,I,\omega,\Omega)$ such that $(M,I)$ is a complex manifold equipped with a K\"{a}hler form $\omega$ and a holomorphic volume form $\Omega$ which satisfy the following relation:
$$\frac{\omega^n}{n!}=(-1)^{\frac{n(n-1)}{2}}\left(\frac{\sqrt{-1}}{2}\right)^n\Omega\wedge\overline{\Omega}.$$
\end{defn}
If $L$ is an oriented Lagrangian submanifold of a Calabi-Yau manifold $(M,I,\omega, \Omega)$, there exists a function $\theta:L\to\mathbb{R}/2\pi\mathbb{Z}$, which is called the \textit{Lagrangian angle} satisfying
$$\iota^\ast\Omega=e^{\sqrt{-1}\theta}\mathrm{vol}_{\iota^\ast g}.$$
Here $g$ is the K\"{a}hler metric, $\iota:L\to M$ is an embedding, and $\mathrm{vol}_{\iota^\ast g}$ is the volume form on $L$ with respect to the induced metric $\iota^\ast g$. Even if $L$ is not orientable, we can locally define the Lagrangian angle with the formula above. With the use of the Lagrangian angle $\theta$ of a Lagrangian submanifold $L$, the mean curvature vector $\mathcal{H}_p$ at $p\in L$ is expressed as follows:
$$\mathcal{H}_p=I_{\iota(p)}(\iota_{\ast p}(\nabla_{\iota^\ast g}\theta)_p)\in\mathrm{T}^\perp_{\iota(p)}\iota(L),$$
where $\nabla_{\iota^\ast g}\theta$ is the gradient of the function $\theta$ with respect to the induced metric $\iota^\ast g$.

The definition of a special Lagrangian submanifold is given by the following:
\begin{defn}
Let $(M,I,\omega,\Omega)$ be a Calabi-Yau manifold. A \textit{special Lagrangian submanifold} of $(M,I,\omega,\Omega)$ is a Lagrangian submanifold such that its Lagrangian angle is constant $\theta\equiv\theta_0$. $\theta_0$ is called the \textit{phase} of the special Lagrangian submanifold.
\end{defn}

From the formula of the mean curvature vector above, we can see that a special Lagrangian submanifold is a minimal submanifold. More strongly it is known that a special Lagrangian submanifold is homologically volume minimizing. 

\subsection{Group actions and moment maps}\label{Section.B.2}

In this subsection we review the fundamental notions of group actions and moment maps.

Let $H$ be a Lie group which acts on $M$. We denote the translation of $h\in H$ by $L_h:M\to M$. For each $p\in M$, the orbit and the isotropy subgroup at $p$ are denoted by $H\cdot p$ and $H_p$ respectively.

Letting $\mathfrak{h}$ denote the Lie algebra of $H$,  any $\xi\in\mathfrak{h}$ induces a fundamental vector field $\xi^\#$ on $M$, defined as follows:
$$\xi^\#_p=\left.\frac{d}{dt}\right|_{t=0}\exp(t\xi)p\quad (p\in M)$$
where $\exp(t\xi)$ denotes the $1$-parameter subgroup of $H$ associated to $\xi$.

$H$ acts on $\mathfrak{h}^\ast$ by the \textit{coadjoint action}:
$$\mathrm{Ad}^\ast_h:\mathfrak{h}^\ast\to\mathfrak{h}^\ast,$$
where $h\in H$, and for $c\in\mathfrak{h}^\ast$, $\mathrm{Ad}^\ast_hc$ is defined as follows:
$$\langle\mathrm{Ad}^\ast_hc,\xi\rangle=\langle c,\mathrm{Ad}_h\xi\rangle\quad (\xi\in\mathfrak{h}).$$
Here $\langle\cdot,\cdot \rangle$ is the pairing of $\mathfrak{h}$ and $\mathfrak{h}^\ast$. We call 
$$Z(\mathfrak{h}^\ast)=\{c\in\mathfrak{h}^\ast\mid\mathrm{Ad}_h^\ast c=c, h\in H\}$$
the \textit{center} of $\mathfrak{h}^\ast$. If $H$ is abelian, then $Z(\mathfrak{h}^\ast)=\mathfrak{h}^\ast$ holds.

\begin{defn}
Let $H$ be a Lie group acting on a symplectic manifold $(M,\omega)$. A \textit{moment map} $\mu :M\to\mathfrak{h}^\ast$ is an $H$-equivariant map that satisfies for any $\xi\in\mathfrak{h}$ the following: 
$$-\mathfrak{i}(\xi^\#)\omega=d\langle\mu(\cdot),\xi\rangle,$$
where $\mathfrak{i}$ is the interior product.
\end{defn}

If $(M,\omega,H)$ has a moment map, the $H$-action is called \textit{Hamiltonian}. A Hamiltonian action preserves $\omega$.
For each $p\in\mu^{-1}(c)$, $c\in\mathfrak{h}^\ast$, the orbit $H\cdot p$ is isotropic if and only if $c\in Z(\mathfrak{h}^\ast)$.

\section{Transformations of holomorphic volume forms}\label{Section.C}

In this section, we retain the notation as in Section \ref{Section.B}. We show a formula (Proposition \ref{Claim.C}) corresponding to transformations of holomorphic volume forms $L_h^\ast\Omega$. We use this formula to calculate the Lagrangian angle of a Lagrangian immersion which we finally construct in Theorem \ref{Claim.J}.

Let $(M,I)$ be a complex manifold and $\Omega$ a holomorphic volume form on $M$. Let $H$ be a Lie group which acts on $M$ preserving $I$. Then the map
$$(L_h)^*:A^k(M)^\mathbb{C}\to A^k(M)^\mathbb{C};\quad \omega\mapsto L_h^*\omega$$
preserves types of complex differential $k$-forms ($k\in\mathbb{N}$), where $A^k(M)^\mathbb{C}$ is the complex vector space which consists of all complex $k$-forms on $M$.
Hence $L_h^\ast\Omega$ is an $(n,0)$-form. Therefore there exists a holomorphic function $f_h$ that satisfies $L_h^\ast\Omega = f_h\Omega$.

Next we introduce a Calabi-Yau structure into a complex manifold, and assume that an $H$-action preserves the K\"{a}hler structure. Then we can see that the holomorphic function $f_h$ satisfies $|f_h|\equiv1$ as follows.

\begin{prop}\label{Claim.B}
Let $(M, I, \omega, \Omega)$ be a $2n$-dimensional Calabi-Yau manifold and $H$ a Lie group which acts on $M$ preserving $I$ and $\omega$.
Then $f_h$ satisfies $|f_h|\equiv1$ on $M$.
\end{prop}
\begin{proof}
The quadruple $(M, {(L_{h*})^{-1}\circ I\circ (L_{h*})}, L_h^*\omega, L_h^*\Omega)$ is also a Calabi-Yau manifold for any $h\in H$, since $H$ preserves $I$ and $\omega$. Therefore, we have
$$
\begin{aligned}
&(-1)^{\frac{n(n-1)}{2}}\left(\frac{\sqrt{-1}}{2}\right)^n \Omega\wedge\overline{\Omega}
=\frac{\omega^n}{n!}
=\frac{(L_h^\ast\omega)^n}{n!}\\
=&(-1)^{\frac{n(n-1)}{2}}\left(\frac{\sqrt{-1}}{2}\right)^n L_h^*\Omega\wedge\overline{L_h^*\Omega}
=(-1)^{\frac{n(n-1)}{2}}\left(\frac{\sqrt{-1}}{2}\right)^n |f_h|^2\Omega\wedge\overline{\Omega}.
\end{aligned}
$$

Comparing the both sides, we obtain $|f_h|\equiv1$.
\end{proof}

By Proposition \ref{Claim.B} we know the following: Because a holomorphic function which has a constant norm on a connected space has to be constant, $f_h\equiv\mathrm{const.}\in U(1)$ on a connected Calabi-Yau manifold. Therefore we can define a map $c:H\to U(1); h\mapsto c(h)=c_h:=f_h$.

The map $c$ is a homomorphism between Lie groups. In fact for $h_1, h_2\in H$, we have
$$
c_{h_2}c_{h_1}\Omega
=L_{h_2}^*(L_{h_1}^*\Omega)
=L_{h_1h_2}^*\Omega
=c_{h_1h_2}\Omega.
$$
Therefore $c_{h_2h_1}=c_{h_1}c_{h_2}=c_{h_2}c_{h_1}$, and $c$ is a homomorphism.

Using this fact, next Proposition \ref{Claim.C} expresses transformations of a holomorphic volume form in a connected Calabi-Yau manifold in terms of a Lie algebra. We assume $H$ to be connected so that we express any $h \in H$ as $h=\exp\eta_1\cdots\exp\eta_l$ by $\eta_1,\cdots,\eta_l\in\mathfrak{h}$. For $h$, such $\eta_1,\cdots,\eta_l\in\mathfrak{h}$ are not unique, however the following holds for any of them.

\begin{prop}\label{Claim.C}
Let $(M, I, \omega, \Omega)$ be a connected Calabi-Yau manifold and $H$ a connected Lie group  which acts on $M$ preserving $I$ and $\omega$.
Then there exists $a_H\in\mathfrak{h}^*$ such that for any $h\in H$, it holds that
$$L_h^*\Omega=e^{\sqrt{-1}\langle a_H, \eta_1+\cdots +\eta_l\rangle}\Omega,$$
where
$$\eta_1,\cdots,\eta_l\in\mathfrak{h}\ \mathrm{such\ that}\ h=\exp \eta_1 \cdots \exp \eta_l.$$
\end{prop}

\begin{proof}
Because $c:H\to U(1)$ defined above is a homomorphism, the following commutative relation holds between  $c$ and $(dc)_e:\mathfrak{h}\cong\mathrm{T}_eH\to\mathfrak{u}(1)$:
$$c\circ \exp \xi = e^{(dc)_e\xi}.$$
In fact, since $c$ makes a one-parameter subgroup $\exp(t\xi)$ of $H$ into a one-parameter subgroup $c(\exp(t\xi))$ of $U(1)$, there exists $\sqrt{-1}\alpha\in\mathfrak{u}(1)$ $(\alpha\in\mathbb{R})$ such that $c(\exp(t\xi))=\exp_{U(1)}(t(\sqrt{-1}\alpha))=e^{\sqrt{-1}t\alpha}$. By differentiating both sides, we obtain
$$
\sqrt{-1}\alpha=\left.\frac{d}{dt}\right|_{t=0}e^{\sqrt{-1}t\alpha}=
\left.\frac{d}{dt}\right|_{t=0}c(\exp(t\xi))=(dc)_e\left.\frac{d}{dt}\right|_{t=0}\exp(t\xi)=(dc)_e\xi.
$$
Thus we see $\sqrt{-1}\alpha=(dc)_e\xi$ and $(c\circ\exp)(t\xi)=e^{t(dc)_e\xi}$. When $t=1$, we obtain $c\circ\exp(\xi)=e^{(dc)_e\xi}$. 

Because $H$ is connected, for each $h\in H$, there exist finite $\eta_1,\cdots,\eta_l\in\mathfrak{h}$ such that $h=\exp \eta_1 \cdots \exp \eta_l$. Then, we have
$$
\begin{aligned}
&c_h
=c(\exp \eta_1\cdots \exp \eta_l)
=c(\exp \eta_1)\cdots c(\exp \eta_l)
=e^{(dc)_e\eta_1}\cdots e^{(dc)_e\eta_l}\\
=&e^{\sqrt{-1}\langle-\sqrt{-1}(dc)_e, \eta_1\rangle}\cdots e^{\sqrt{-1}\langle-\sqrt{-1}(dc)_e, \eta_l\rangle}
=e^{\sqrt{-1}\langle-\sqrt{-1}(dc)_e, \eta_1+\cdots +\eta_l\rangle}.
\end{aligned}
$$
Therefore noting $\mathfrak{u}(1)=\{\sqrt{-1}\varphi\in\mathbb{C}\mid\varphi\in\mathbb{R}\}$ and letting $a_H:=-\sqrt{-1}(dc)_e$, we can define a linear map $a_H:\mathfrak{h}\to\mathbb{R}$, i.e., $a_H\in\mathfrak{h}^\ast$ and the claim of the proposition holds.
\end{proof}

By Proposition \ref{Claim.C}, transformations of a holomorphic volume form are expressed in terms of a Lie algebra. This enables us to explicitly show the Lagrangian angle of a Lagrangian immersion $(H/K)\times V\to M$ which we construct in the next section in terms of the Lie algebra $\mathfrak{h}$ at each $(hK,p)\in (H/K)\times V$. Here $K$ is a closed Lie subgroup of $H$ and $V$ is a submanifold in $M$.

\begin{cor}\label{Claim.E}
Let $(M, I, \omega, \Omega, H)$ be same as Proposition \ref{Claim.C}. Then $a_H=0$ if and only if the $H$-action preserves $\Omega$, namely it preserves the Calabi-Yau structure.
\end{cor}
\section{Special Lagrangian construction}\label{Section.D}
In this section we show a construction of special Lagrangian submanifolds by ``(generalized) perpendicular symmetries'', using the formula (Proposition \ref{Claim.C}) which we proved in the previous section. We construct an isotropic immersion, especially a Lagrangian immersion in Proposition \ref{Claim.H}. We give a formula that express the Lagrangian angle of this Lagrangian immersion in Theorem \ref{Claim.J}. We finally construct a special Lagrangian immersion in Corollary \ref{Claim.M} by considering a condition to have constant Lagrangian angle.
\subsection{Immersions}\label{Section.D.1}

First with the use of group actions, we construct an immersion which is fundamental for our constructions. This immersion has a form $H\cdot V$ for a submanifold $V$ in $M$. When $H$ is abelian, it might be natural to assume that the action is free. Otherwise, we may need to consider singular orbits. To control them, we add a condition that the isotropy subgroup $H_p$ at each point $p\in V$ is a constant $K$. Lemma \ref{Claim.G} is one of important properties that these immersions have. 

\begin{prop}\label{Claim.F}
Let $M$ be a manifold and $H$ a Lie group which acts on $M$. Let $\mathfrak{h}$ be the Lie algebra of $H$ and $V$ a submanifold of $M$.
Assume the followings:
\begin{enumerate}\setlength{\labelsep}{1em}\setlength{\itemindent}{5em}
\item[\em{(Imm-$H$)}]$\xi^\#_p\notin \mathrm{T}_pV\backslash\{0\}$ for any $p\in V$ and any $\xi\in \mathfrak{h}$, and
\item[\em{(Imm-istp)}]the isotropy subgroup at $p$ is a constant $K$ for any $p\in V$.
\end{enumerate}
Define a map $\phi : (H/K)\times V \to M$ by $(hK,p) \mapsto hp$. Then $\phi$ is an immersion.
\end{prop}

\begin{lem}\label{Claim.G}
Assume the conditions of Proposition \ref{Claim.F}. For any $(hK,p)\in (H/K)\times V$, any $\xi\in\mathfrak{h}$, and any $v\in\mathrm{T}_pV$, the following holds:
$${\phi_*}_{(hK,p)}\left(\left.\frac{d}{dt}\right|_{t=0}h \exp (t\xi) K,v\right) = (L_h)_{*p}(\xi^\#_p+v).$$
\end{lem}

\begin{proof}[Proof of Lemma \ref{Claim.G}]
By (Imm-istp), the map $\phi$ is well-defined.

Fix an arbitrary point $(hK,p) \in (H/K)\times V$. First we show the following:
$$\mathrm{T}_{hK}(H/K)=\left\{\left.\frac{d}{dt}\right|_{t=0}h \exp (t\xi) K \,\Bigg|\, \xi\in \mathfrak{h}\right\}.$$
For $g\in H$, define $\tau_g$ by
$$\tau_g:(H/K) \to (H/K);\quad hK \mapsto ghK.$$
The map $\tau_g$ is an element of Diff$(H/K)$, here Diff$(H/K)$ is the space of all diffeomorphisms on $H/K$.
We have $\mathrm{T}_{K}(H/K)=\{\left.\frac{d}{dt}\right|_{t=0}\exp (t\xi) K\ |\  \xi\in\mathfrak{h}\}$. We also have
$$(\tau_h)_{*K}\left.\frac{d}{dt}\right|_{t=0}\exp (t\xi) K = \left.\frac{d}{dt}\right|_{t=0}h \exp (t\xi) K.$$
The linear map $(\tau_h)_{*K}:\mathrm{T}_K(H/K) \to T_{hK}(H/K)$ is an isomorphism. Therefore the claim above holds.

Let $\gamma(t)$ be a curve in $V$ that satisfies $\gamma(0)=p$ and $\gamma^\prime(0)=v$. Then we have
$$(\phi_*)_{(hK,p)}\left(\left.\frac{d}{dt}\right|_{t=0}h\exp (t\xi) K,0\right)=\left.\frac{d}{dt}\right|_{t=0}h\exp (t\xi) p = (L_h)_{*p}\xi^\#_p,$$
$$(\phi_*)_{(hK,p)}(0,v) = \left.\frac{d}{dt}\right|_{t=0}h\gamma(t) = (L_h)_{*p}v.$$
Thus Lemma \ref{Claim.G} has been proved.
\end{proof}
\begin{proof}[Proof of Proposition \ref{Claim.F}]
To prove Proposition \ref{Claim.F}, it is sufficient to show that if for any $\xi\in\mathfrak{h}$ and $v\in\mathrm{T}_pV$, 
$${\phi_*}_{(hK,p)}\left(\left.\frac{d}{dt}\right|_{t=0}h \exp (t\xi) K,v\right) = (L_h)_{*p}(\xi^\#_p+v) = 0,$$
then
$$\left.\frac{d}{dt}\right|_{t=0}h \exp (t\xi) K=0,\quad v=0.$$
Since $(L_h)_{\ast p}$ is an isomorphism, if $(L_h)_{\ast p}(\xi^\#_p +v)=0,$ then $\xi^\#_p+v=0$.
 By (Imm-$H$), a pair $(\xi^\#_p, v)$ is linearly independent. Hence we have $\xi^\#_p=0$ and $v=0$ from $\xi^\#_p+v=0$.
If we define a map $j$ by
$$j: (H/K) \to (H\cdot p);\quad hK \mapsto hp,$$
the map $j$ is a diffeomorphism.
With the isomorphism ${j_*}_{hK}:\mathrm{T}_{hK}(H/K)\to\mathrm{T}_{hp}(H\cdot p)$, we have
$$\left.\frac{d}{dt}\right|_{t=0}h\exp (t\xi) K \mapsto \left.\frac{d}{dt}\right|_{t=0}h\exp (t\xi) p = {L_h}_{*p}\xi^\#_p.$$
Thus $\left.\frac{d}{dt}\right|_{t=0}h\exp (t\xi) K=0$ if and only if $\xi^\#_p=0$.
\end{proof}
\subsection{Isotropic immersions}\label{Section.D.2}
Next we introduce a symplectic structure to a manifold $M$, and show a condition for the immersions of Proposition \ref{Claim.F} to be isotropic in Proposition \ref{Claim.H}.
\begin{prop}\label{Claim.H}
Let $(M,\omega)$ be a $2n$-dimensional symplectic manifold and $H$ a Lie group which acts on $M$ and has a moment map $\mu$. Let $\mathfrak{h}$ be the Lie algebra of $H$ and $c$ an element of $\mathfrak{h}^\ast$. Let $V_c$ be a submanifold of $M$ that satisfies $V_c\subset\mu^{-1}(c)$.

Assume \em{(Imm-$H$)}, \em{}\em{(Imm-istp)}\em{}, and the followings:
\begin{enumerate}\setlength{\labelsep}{1em}\setlength{\itemindent}{5em}
\item[(\em{Istp-$V_c$)}] $V_c$ is isotropic, and
\item[\em{(Istp-cnt)}]$c$ is an element of $Z(\mathfrak{h}^\ast)$, the center of $\mathfrak{h}^\ast$.
\end{enumerate}
Define a map $\phi_c : (H/K)\times V_c \to M$ by $(hK,p) \mapsto hp$. Then $\phi_c$ is an isotropic immersion.

In addition, if the following condition holds, $\phi_c$ is a Lagrangian immersion:
\begin{enumerate}\setlength{\labelsep}{1em}\setlength{\itemindent}{5em}
\item[\em{(Lag-dim)}]$\dim H/K+\dim V_c=n$.
\end{enumerate}
\end{prop}

\begin{lem}\label{Claim.I}
Assume the settings of Proposition \ref{Claim.H} except the conditions \em{(Imm-$H$)}\em{}, \em{(Imm-istp)}\em{}, \em{(Istp-$V_c$)}\em{}, \em{(Istp-cnt)}\em{}, and \em{(Lag-dim)}\em{}. Then $\omega_p(\xi^\#_p, v)=0$ for any $p\in V_c$, any $v\in\mathrm{T}_pV_c$, and any $\xi\in\mathfrak{h}$.
\end{lem}

\begin{proof}[Proof of Lemma \ref{Claim.I}]
Noting $(d\mu)_pv=0$, we have
$$
\begin{aligned}
\omega_p(\xi^\#_p, v)
=-d(\langle\mu(\cdot), \xi\rangle)_p v
=-\langle(d\mu)_pv,\xi\rangle
=0.
\end{aligned}
$$
Thus we have shown Lemma \ref{Claim.I}.
\end{proof}
\begin{proof}[Proof of Proposition \ref{Claim.H}]
Since the map $\phi_c$ is an immersion by Proposition \ref{Claim.F}, it is sufficient for proving Proposition \ref{Claim.H} to show that ${\phi_c^*}\omega \equiv 0$ on $(H/K)\times V_c$.
For two arbitrary elements $(\left.\frac{d}{dt}\right|_{t=0}h\exp (t\xi_i) K,v_i)\in \mathrm{T}_{hK}(H/K)\times\mathrm{T}_pV_c (i=1, 2)$, we have
$$
\begin{aligned}
&(\phi_c^*\omega)_{(hK,p)}\left((\left.\frac{d}{dt}\right|_{t=0}h\exp (t\xi_1) K,v_1), (\left.\frac{d}{dt}\right|_{t=0}h\exp (t\xi_2) K,v_2)\right)\\
=&\omega_{hp}\left((\phi_c)_{*(hK,p)}(\left.\frac{d}{dt}\right|_{t=0}h\exp (t\xi_1) K,v_1), (\phi_c)_{*(hK,p)}(\left.\frac{d}{dt}\right|_{t=0}h\exp (t\xi_2) K,v_2)\right)\\
=&\omega_{hp}\bigg((L_h)_{*p}\{(\xi_1)^\#_p+v_1\}, (L_h)_{*p}\{(\xi_2)^\#_p+v_2\}\bigg)\\
=&\omega_p((\xi_1)^\#_p+v_1, (\xi_2)^\#_p+v_2)\\
=&\omega_p((\xi_1)^\#_p, (\xi_2)^\#_p)+\omega_p((\xi_1)^\#_p, v_2)+\omega_p(v_1, (\xi_2)^\#_p)+\omega_p(v_1, v_2).
\end{aligned}
$$
The first term is equal to zero by (Istp-cnt), the second and third terms are zero by Lemma \ref{Claim.I}, and the forth term is zero by (Istp-$V_c$). Thus we see that $\phi$ is an isotropic immersion. In addition, if (Lag-dim) holds,  this immersion is Lagrangian by the definition of Lagrangian submanifolds.
\end{proof}

\subsection{Lagrangian angle and special Lagrangian construction}\label{Section.D.3}

We constructed a Lagrangian immersion in Proposition \ref{Claim.H}. We show a condition for this immersion to be a special Lagrangian immersion by using the Lagrangian angle. In Theorem \ref{Claim.J}, with the use of a formula for transformations of holomorphic volume forms (Proposition \ref{Claim.C}), we give explicitly the Lagrangian angle of a Lagrangian immersion of Proposition \ref{Claim.H}.

Lemma \ref{Claim.K} is used for calculations of the Lagrangian angle.

\begin{thm}\label{Claim.J}
Let $(M,g,I,\omega,\Omega)$ be a connected $2n$-dimensional Calabi-Yau manifold and $H$ a connected Lie group which acts on $M$ preserving $I$ and has a moment map $\mu$. Let $\mathfrak{h}$ be the Lie algebra of $H$ and $L$ a Lagrangian submanifold of $M$ that has a local Lagrangian angle $\theta$. Let $c$ be an element of $\mathfrak{h}^\ast$ and $V_c$ a submanifold of $M$ that satisfies $V_c\subset\mu^{-1}(c)\cap L$. Assume \em{(Imm-istp)}\em{}, \em{(Istp-cnt)}\em{}, \em{(Lag-dim)}\em{}, and the following \em{(LagAng-$H$)}\em{}:
\begin{enumerate}\setlength{\labelsep}{1em}\setlength{\itemindent}{5em}
\item[\em{(LagAng-$H$)}]For any $p\in V_c$ and any $\xi\in\mathfrak{h}$, the following \em{(i)}\em{} and \em{(ii)}\em{} hold:
\begin{enumerate}\setlength{\labelsep}{1em}\setlength{\itemindent}{7em}
\item[\em{(i)}]$\xi^\#_p\in\mathrm{T}^\perp_pL\oplus\mathrm{T}_pV_c$,
\item[\em{(ii)}]$\xi^\#_p\notin\mathrm{T}_pV_c\backslash\{0\}$.
\end{enumerate}
\end{enumerate}
Define $\phi_c : (H/K)\times V_c \to M$ as in Proposition \ref{Claim.H}.
Then locally the following holds:
$$(\phi_c^*\Omega)_{(hK,p)} = \pm e^{\sqrt{-1}\theta_c}\mathrm{vol}_{\phi_c^*g},$$
where
$$
\begin{aligned}
\theta_c(hK,p) =\langle a_H, \eta_1+\cdots +\eta_l\rangle + \theta(p) -\frac{\pi}{2}\dim(H/K),\\
\eta_1,\cdots,\eta_l \in\mathfrak{h} \mathrm{\ such\ that\ } h=\exp \eta_1\cdots \exp \eta_l.
\end{aligned}
$$
\end{thm}

\begin{lem}\label{Claim.K}
Under the conditions of Theorem \ref{Claim.J}, for any $p\in V_c$ there exist
$$\xi_1,\cdots,\xi_m\in\mathfrak{h},\quad v_1,\cdots,v_{n-m},w_1,\cdots,w_m\in\mathrm{T}_pV_c$$
that satisfy the followings:
\begin{enumerate}\setlength{\labelsep}{1em}\setlength{\itemindent}{1em}
\item[(1)]For any $h\in H$,
$$
\begin{aligned}
&\left(
\cdots, (\left.\frac{d}{dt}\right|_{t=0}h\exp (t\xi_j)K,w_j),\cdots,
 \cdots (0,v_i), \cdots
\right)\\
&\hspace{4em}(i=1,\cdots, n-m,\ j=1,\cdots,m)
\end{aligned}
$$
is an orthonormal basis in $\mathrm{T}_{(hK,p)}((H/K)\times V_c)$ with respect to $\phi_c^\ast g$, 
\item[(2)]$(\xi_j)^\#_p+w_j\in\mathrm{T}_p^\perp L$ for $j=1,\cdots,m$, and
\item[(3)]$\bigg(I_p\{(\xi_1)^\#_p+w_1\},\cdots,I_p\{(\xi_m)^\#_p+w_m\},v_1,\cdots,v_{n-m}\bigg)$ is an orthonormal basis in $\mathrm{T}_pL$ with respect to $\iota^\ast g$.
\end{enumerate}
Here $m=\dim(H/K)$, and $\iota:L\to M$ is the embedding.
\end{lem}

\begin{rem}\label{Claim.L}
In Theorem \ref{Claim.J}, we do not assume the conditions (Imm-$H$) and (Istp-$V_c$) in Proposition \ref{Claim.H} to make $\phi_c$ a Lagrangian immersion. However under the conditions of Theorem \ref{Claim.J}, they hold.
\end{rem}

From Theorem \ref{Claim.J} we immediately obtain the following corollary. Constructions of special Lagrangian submanifolds are directly based on this corollary.

\begin{cor}\label{Claim.M}
Assume the conditions of Theorem \ref{Claim.J}. In addition, if $\theta\equiv\mathrm{const.}$ on $V_c$ (e.g. $L$: a special Lagrangian submanifold) and $a_H = 0$, then $\phi_c$ is a special Lagrangian immersion.
\end{cor}

Now we prove Lemma \ref{Claim.K}, Theorem \ref{Claim.J}, and Remark \ref{Claim.L}.

\begin{proof}[Proof of Lemma \ref{Claim.K}]
First we show Lemma \ref{Claim.K} (1) and (2).
By Lemma \ref{Claim.G}, we have
{\small$$
\begin{aligned}
\left(\phi_{c\ast(hK,p)}(\left.\frac{d}{dt}\right|_{t=0}h\exp (t\xi_j)K,w_j),\phi_{c\ast(hK,p)}(0,v_i)\right)
=(L_{h\ast p}\{(\xi_j)^\#_p+w_j\},L_{h\ast p}v_i).
\end{aligned}
$$}
Since $L_{h\ast p}$ is isometric, it is enough for showing (1) and (2) to verify that there exist $\xi_j$, $v_i$, and $w_j$ $(i=1,\cdots, n-m,\ j=1,\cdots, m)$ such that $((\xi_j)^\#_p+w_j,v_i)$ is an orthonormal system of $\mathrm{T}_pM$ and $(\xi_j)^\#_p+w_j\in\mathrm{T}^\perp_pL$.

Noting (Lag-dim), let $(v_i)\ (i=1,\cdots,n-m)$ be an orthonormal basis of $\mathrm{T}_pV_c$ with respect to the metric on $V$ induced from $g$. By (LagAng-$H$), it holds that $\mathrm{T}_p(H\cdot p)\cap\mathrm{T}_pL=\{0\}$. Hence, noting (Lag-dim) again, we can take $\eta_j \in\mathfrak{h}\ (j=1,\cdots,m)$ such that $((\eta_j)^\#_p)$ is a basis of $\mathrm{T}_p(H\cdot p)$ and $((\eta_1)^\#_p,\cdots,(\eta_m)^\#_p,v_1,\cdots,v_{n-m})$ is linearly independent in $\mathrm{T}_pM$.

By (LagAng-$H$), there exist $u_j\in\mathrm{T}_p^\perp L\backslash\{0\}$ and $z_j\in\mathrm{T}_pV_c\ (j=1,\cdots, m)$ that decompose $(\eta_j)^\#_p$ into direct summations as follows:
$$(\eta_j)^\#_p=u_j+z_j\quad (j=1,\cdots,m).$$

$(u_j)$ is linearly independent. In fact, if $u_1$ is expressed by $u_1=b_2u_2+\cdots +b_mu_m$ for $b_j\in\mathbb{R}$ such that ${}^t(b_2,\cdots,b_m)\neq\mathbf{0}$, we have
$$
\begin{aligned}
&(\eta_1)^\#_p-z_1=b_2((\eta_2)^\#_p-z_2)+\cdots +b_m((\eta_m)^\#_p-z_m)\\
\Leftrightarrow\ &(\eta_1)^\#_p-\{b_2(\eta_2)^\#_p+\cdots +b_m(\eta_m)^\#_p\}
=z_1-(b_2z_2+\cdots +b_mz_m).
\end{aligned}
$$
Because the left-hand side belongs to $\mathrm{T}_p(H\cdot p)$, there exists $\eta\in\mathfrak{h}$ such that $\eta^\#_p$ equals to the left-hand side. If $\eta^\#_p\neq 0$, then $\eta^\#_p\in\mathrm{T}_pV_c\backslash\{0\}$ because the right-hand side belongs to $\mathrm{T}_pV_c$. This is contrary to (LagAng-$H$). On the other hand, if $\eta^\#_p=0$, this is contrary to that $(\eta_j)^\#_p$ is linearly independent because the left-hand side equals to $0$. For $j=2,\cdots,m$, we can verify the same assertion. Thus $(u_j)$ is linearly independent.

Therefore, noting $u_j\in\mathrm{T}^\perp_pL$, there exists $A\in GL(m,\mathbb{R})$ such that $(\tilde{u}_1\cdots, \tilde{u}_m) = (u_1,\cdots, u_m)A$ is an orthonormal system in $\mathrm{T}^\perp_pL$. Because $\mathrm{T}^\perp_pL\perp \mathrm{T}_pV_c$, $(v_i,\tilde{u}_j)$ is an orthonormal system in $\mathrm{T}_pM$.

Thus, if we define $\xi_j\in\mathfrak{h}$ and $w_j\in\mathrm{T}_pV_c$ by $((\xi_1)^\#_p,\cdots,(\xi_m)^\#_p)=((\eta_1)^\#_p,\cdots,(\eta_m)^\#_p)A$ and $(w_1,\cdots, w_m)=(-z_1,\cdots,-z_m)A$, then $\tilde{u}_j=(\xi_j)^\#_p+w_j$ and Lemma \ref{Claim.K} (1) and (2) hold.

Next we show Lemma \ref{Claim.K} (3). 
By (1) and (2), it is enough for showing the claim of (3) to verify $I_p\{(\xi_j)^\#_p+w_j\}\in\mathrm{T}^\perp_p V_c\ (j=1,\cdots,m)$.

$I_p(\xi_j)^\#_p\in\mathrm{T}^\perp_pV_c$ because $0=\omega_p((\xi_j)^\#_p,v_i)=g_p(I_p(\xi_j)^\#_p,v_i)$ by Lemma \ref{Claim.I}. 
On the other hand, $I_pw_j\in\mathrm{T}^\perp_pV_c$ because $V_c$ is isotropic and $0=\omega_p(w_j,v_i)=g_p(I_pw_j,v_i)$. Thus Lemma \ref{Claim.K} (3) has been verified.
\end{proof}
\begin{proof}[Proof of Theorem \ref{Claim.J}]
Let $\mathcal{X}^{(0,1)}(M)$ be the set of complex vector fields of type $(0,1)$ on $M$. For any $\eta\in\mathfrak{h}$, it holds that $\eta^\#+\sqrt{-1}I\eta^\#\in\mathcal{X}^{(0,1)}(M)$. Since $\Omega$ is a complex differential form of type $(n,0)$ on $M$, we have
$$\mathfrak{i}(\eta^\#)\Omega=-\sqrt{-1}\mathfrak{i}(I\eta^\#)\Omega.$$

Take $(\xi_j,v_i,w_j)$ in Lemma \ref{Claim.K} for $i=1,\cdots, n-m$ and $j=1,\cdots, m$. Then, noting $(L^\ast_h\Omega)=e^{\sqrt{-1}\langle a_H,\eta_1+\cdots +\eta_l\rangle}\Omega$, we have
{\small$$
\begin{aligned}
&(\phi_c^*\Omega)_{(hK,p)}\left(\cdots, (\left.\frac{d}{dt}\right|_{t=0}h\exp (t\xi_j) K,w_j), \cdots, \cdots, (0,v_i), \cdots\right)\\
=&\Omega_{hp}(\cdots, (L_h)_{*p}\{(\xi_j)^\#_p+w_j\}, \cdots, \cdots, (L_h)_{*p}v_i, \cdots)\\
=&(L_h^*\Omega)_p(\cdots, (\xi_j)^\#_p+w_j, \cdots, \cdots, v_i, \cdots)\\
=&(-\sqrt{-1})^m(L_h^*\Omega)_p(\cdots, I_p\{(\xi_j)^\#_p+w_j\}, \cdots, \cdots, v_i, \cdots)\\
=&(-\sqrt{-1})^me^{\sqrt{-1}\langle a_H, \eta_1+\cdots +\eta_l\rangle}\Omega_p(\cdots, I_p\{(\xi_j)^\#_p+w_j\}, \cdots, \cdots, v_i, \cdots)\\
=&(-\sqrt{-1})^me^{\sqrt{-1}\langle a_H, \eta_1+\cdots +\eta_l\rangle}(\iota^*\Omega)_p(\cdots, I_p\{(\xi_j)^\#_p+w_j\}, \cdots, \cdots, v_i, \cdots)\\
=&(-\sqrt{-1})^me^{\sqrt{-1}\langle a_H, \eta_1+\cdots +\eta_l\rangle}e^{\sqrt{-1}\theta}(\mathrm{vol}_{\iota^*g})_p(\cdots, I_p\{(\xi_j)^\#_p+w_j\}, \cdots, \cdots, v_i, \cdots)\\
=&\pm e^{\sqrt{-1}(\langle a_H, \eta_1+\cdots +\eta_l\rangle+\theta-\frac{\pi}{2}m)}.
\end{aligned}
$$}
By Lemma \ref{Claim.K} (1), we have
$$\mathrm{vol}_{\phi_c^\ast g}\left(\cdots, (\left.\frac{d}{dt}\right|_{t=0}h\exp (t\xi_j) K,w_j), \cdots, \cdots, (0,v_i), \cdots\right)=\pm 1.$$
Thus Theorem \ref{Claim.J} has been proved.
\end{proof}
\begin{proof}[Proof of Remark \ref{Claim.L}]
(Imm-$H$) holds by (LagAng-$H$).
Because $L$ is a Lagrangian submanifold and $V_c\subset L$, (Istp-$V_c$) holds.
Thus we see that Remark \ref{Claim.L} holds.
\end{proof}

Joyce pointed out in \cite{Joyce2002} that the commutativity of a Lie group is a necessary condition for the group action to be perpendicular to whole of $L$. However, to construct a special Lagrangian submanifold, we need for a group action to be perpendicular only to $V_c\subset L$. Similarly the condition that $L$ is a special Lagrangian submanifold, that is, the condition that the Lagrangian angle is constant on $L$ is reduced to on $V_c$. The perpendicular condition is also weakened as above. This situation, roughly speaking, indicates that a special Lagrangian submanifold may be constructed by Corollary \ref{Claim.M}, if $H\cdot V_c$ (not necessarily each fundamental vector $\xi^\#_p$ at $p\in V_c$) is perpendicular to $L$ for some $c\in Z(\mathfrak{h}^\ast)$.

As a special case of the condition (LagAng-$H$) if we assume that each fundamental vector $\xi^\#_p$ is perpendicular to $L$, we obtain the next corollary. In this case we need not assume (Istp-cnt) (see Remark \ref{Claim.O}).

\begin{cor}\label{Claim.N}
Let $(M,g,I,\omega,\Omega)$ be a connected $2n$-dimensional Calabi-Yau manifold and $H$ a connected Lie group which acts on $M$ preserving $I$ and has a moment map $\mu$. Let $\mathfrak{h}$ be the Lie algebra of $H$ and $L$ a Lagrangian submanifold of $M$ with a local Lagrangian angle $\theta$. Let $c$ be an element of $\mathfrak{h}^\ast$ and $V_c$ a submanifold of $M$ such that $V_c\subset\mu^{-1}(c)\cap L$. Assume \em{(Imm-istp)}\em{}, \em{(Lag-dim)}\em{}, and \em{(LagAng-$H$)$^\prime$}\em{} as follows:
\begin{enumerate}\setlength{\labelsep}{1em}\setlength{\itemindent}{5em}
\item[\em{(LagAng-$H$)$^\prime$}]$\xi^\#_p\perp\mathrm{T}_pL$ for any $p\in V_c$, and any $\xi\in\mathfrak{h}$.
\end{enumerate}
Define $\phi_c:(H/K)\times V_c\to M$ as in Proposition \ref{Claim.H}. Then locally the following holds:
$$(\phi_c^*\Omega)_{(hK,p)} = \pm e^{\sqrt{-1}\theta_c}\mathrm{vol}_{\phi_c^*g},$$
where
$$
\begin{aligned}
\theta_c(hK,p) =\langle a_H, \eta_1+\cdots +\eta_l\rangle + \theta(p) -\frac{\pi}{2}\dim(H/K),\\
\eta_1,\cdots,\eta_l \in\mathfrak{h}\ \mathrm{such\ that}\ h=\exp \eta_1\cdots\exp \eta_l.
\end{aligned}
$$
\end{cor}

\begin{rem}\label{Claim.O}
Under the conditions of Corollary \ref{Claim.N}, (Imm-$H$), (Istp-$V_c$), and (Istp-cnt) hold.
\end{rem}

\begin{proof}[Proof of Remark \ref{Claim.O}]
 (Imm-$H$) holds by (LagAng-$H$)$^\prime$. (Istp-$V_c$) holds as in Remark \ref{Claim.L}. Finally to show (Istp-cnt), we fix an arbitrary point $hp\in H\cdot p\ (h\in H)$. We have
$$\mathrm{T}_{hp}(H\cdot p) = \left.\left\{\left.\frac{d}{dt}\right|_{t=0}h\exp (t\xi) p \ \right|\  \xi\in\mathfrak{h}\right\}=\{(L_h)_{\ast p}\xi^\#_p\mid \xi\in\mathfrak{h}\}.$$
Noting $I_p\xi^\#_p\in\mathrm{T}_pL$ and $\eta^\#_p\in\mathrm{T}^\perp_p L$ because
of the assumption that $L$ is Lagrangian and (LagAng-$H$)$^\prime$, we have
$$
\omega_{hp}((L_h)_{*p}\xi^\#_p, (L_h)_{*p}\eta^\#_p)
=(L_h^*\omega)_p(\xi^\#_p, \eta^\#_p)
=g_p(I_p\xi^\#_p, \eta^\#_p)
=0.
$$
Therefore $H\cdot p$ is isotropic. This is equivalent to $\mu(p)\in Z(\mathfrak{h^\ast})$.
\end{proof}
\begin{cor}\label{Claim.Z}
Assume the conditions of Corollary \ref{Claim.N}. In addition, if $\theta\equiv\mathrm{const.}$ on $V_c$ (e.g. $L$: a special Lagrangian submanifold) and $a_H = 0$, then $\phi_c$ is a special Lagrangian immersion.
\end{cor}
\section{Examples in $\mathrm{T}^\ast S^n$}\label{Section.E}

In this section, with the use of the results above, we construct non-trivial examples of special Lagrangian submanifolds in non-flat Calabi-Yau manifolds $\mathrm{T}^\ast S^n$ which equipped with the Stenzel metrics.
In Subsection \ref{Section.E.1}, we review some notions about the Stenzel metrics on $\mathrm{T}^\ast S^n$, and make sure some facts that is used to construct our examples. In Subsection \ref{Section.E.3}, we construct two examples by using the actions of an abelian group. One of them is based on Corollary \ref{Claim.J} of generalized perpendicular conditions.
In Subsection \ref{Section.E.2}, we construct an example based on Corollary \ref{Claim.Z} by using the actions of a non-abelian group.

Through this section, we use some notations. We denote $\mathbf{e}_i$ by the column $k$-vector whose $i$-th element equals to $1$ and the any other element equals to $0$ in $k$-dimensional real or complex Euclidean space for some $k\in\mathbb{N}$. We also denote $\xi_{ij}$ by
$$\xi_{ij} = E_{ji}-E_{ij} \in M(k,\mathbb{R}),$$
here $E_{ij}$ denotes the $k\times k$-matrix whose $(i,j)$-component is $1$ and all the others are $0$ for some $k\in\mathbb{N}$.
\subsection{Stenzel metric on $\mathrm{T}^\ast S^n$}\label{Section.E.1}

In \cite{Stenzel}, Stenzel constructed complete Ricci-flat K\"{a}hler metrics on the cotangent bundles of compact rank one symmetric spaces. This gives us examples of non-flat Calabi-Yau manifolds. In this paper, we denote this Calabi-Yau structure by $(\mathrm{T}^\ast S^n, I, \omega_{\mathrm{Stz}},\Omega_{\mathrm{Stz}})$. We construct our examples of special Lagrangian submanifolds in $(\mathrm{T}^\ast S^n, I, \omega_{\mathrm{Stz}},\Omega_{\mathrm{Stz}})$.

We identify the tangent bundle and the cotangent bundle of the $n$-sphere $S^n$, and describe it by
$$\mathrm{T}^\ast S^n
=
\{
(x,\xi)\in\mathbb{R}^{n+1}\times\mathbb{R}^{n+1}\mid\|x\|=1, x\centerdot\xi=0
\},
$$
where ``\ $\centerdot$\ ''\ is the canonical real inner product on the Euclidean space $\mathbb{R}^{n+1}$ and $\|x\|=\sqrt{x\centerdot x}$ for each $x\in\mathbb{R}^{n+1}$. We occasionally denote ${}^t(x_1,\cdots,x_{n+1}), {}^t(\xi_1,\cdots,\xi_{n+1})$ by $x, \xi$ respectively.
$SO(n+1)$ acts on $\mathrm{T}^\ast S^n$ by $h\cdot(x,\xi)=(hx,h\xi)$ for $h\in SO(n+1)$ with cohomogeneity one. The principal orbit at a point $(x,\xi)$ equals to a sphere bundle with a radius of $\|\xi\|=\sqrt{\xi\centerdot\xi}$.

Let $Q^n$ be a complex quadric hypersurface in $\mathbb{C}^{n+1}$ as follows:
$$Q^n=\left\{z={}^t(z_1,\cdots,z_{n+1})\in\mathbb{C}^{n+1}\ \left|\ \sum_{i=1}^{n+1}z_i^2=1\right\}\right..$$
Sz\"{o}ke gave an $SO(n+1)$-equivariant diffeomorphism from $\mathrm{T}^\ast S^n$ to $Q^n$ in \cite{Szoke} as follows:
$$
\begin{array}{cccl}
\Phi:&\mathrm{T}^\ast S^n&\to&Q^n\\
&\rotatebox{90}{$\in$}&&\rotatebox{90}{$\in$}\\
&(x,\xi)&\mapsto&\cosh(\|\xi\|)x+\sqrt{-1}\dfrac{\sinh(\|\xi\|)}{\|\xi\|}\xi.
\end{array}
$$
We can induce a complex structure to $Q^n$ from $\mathbb{C}^{n+1}$. Stenzel constructed Ricci-flat K\"{a}hler metrics with respect to these complex structures. We denoted this by $I$ above. Therefore when we use the complex structure for studying the perpendicular conditions later, we do the calculations not in $\mathrm{T}^\ast S^n$ but in $Q^n$. The K\"{a}hler form $\omega_{Stz}$ that Stenzel constructed is given as follows:
$$
\omega_{\mathrm{Stz}}=\sqrt{-1}\partial\bar{\partial}u(r^2)=\sqrt{-1}\sum_{i,j=1}^{n+1}\frac{\partial^2}{\partial z_i\partial\bar{z}_j}u(r^2)dz_i\wedge d\bar{z}_j,
$$
here $r^2=\|z\|^2=\sum_{i=1}^{n+1}z_i\bar{z}_i$ and $u$ is a smooth real function satisfies the following ordinary differential equation:
\begin{equation}\label{to}
\frac{d}{dt}(U^\prime(t))^n=cn(\sinh t)^{n-1}\quad (c=\mathrm{const.}>0),
\end{equation}
here $U(t)=u(\cosh t)$.
The functions $U$ and $u$ has properties that $U^\prime(t)>0$, $U^{\prime\prime}(t)>0$, and $u^\prime(t)>0$ if $t>0$ under appropriate choices of a constant of integration (see \cite{Stenzel}).

The actions of $SO(n+1)$ preserve the Calabi-Yau structure of $(\mathrm{T}^\ast S^n,I,\omega_{\mathrm{Stz}},\Omega_{\mathrm{Stz}})$. Hence, for $a_H = a_{SO(n+1)}\in\mathfrak{h}^\ast =\mathfrak{so}(n+1)^\ast$ of Proposition \ref{Claim.C} determined by $(\mathrm{T}^\ast S^n, I, \omega_{\mathrm{Stz}}, \Omega_{\mathrm{Stz}}, SO(n+1))$, we have $a_H =0$ by Corollary \ref{Claim.E}.

A moment map $\mu:Q^n\to\mathfrak{so}(n+1)^\ast$ with respect to $(\mathrm{T}^\ast S^n, \omega_\mathrm{Stz}, SO(n+1))$ is given in \cite{Anciaux} as follows:
\begin{equation}\label{na}
(\mu(z))(X)=u^\prime(r^2)Iz\centerdot Xz,\quad (z\in Q^n, X\in\mathfrak{so}(n+1)),
\end{equation}
here ``\ $\centerdot$\ '' denotes the canonical real inner product on $\mathbb{C}^{n+1}$.

Finally, we give a basic fact for preparing an original special Lagrangian submanifold to construct a new one: Karigiannis and Min-Oo showed in \cite{Karigiannis and Min-Oo} that a conormal bundle $\mathrm{T}^{\ast\perp}N$ in $\mathrm{T}^\ast S^n$ for a submanifold $N$ in $S^n$ is a special Lagrangian submanifold if and only if $N$ is an austere submanifold of $S^n$. Especially, a totally geodesic submanifold of $S^n$ is an austere submanifold.

\subsection{The case of $H=U(1), L_1=\mathrm{T}^{\ast\perp}S^2, L_2=\mathrm{T}^{\ast\perp}S^1\subset\mathrm{T}^\ast S^5$}\label{Section.E.3}

Let $M$ be the cotangent bundle of $5$-sphere $\mathrm{T}^\ast S^5$, $L_1$ the conormal bundle of a totally geodesic submanifold $S^2$ of $S^5$, and $L_2$ the conormal bundle of a totally geodesic submanifold $S^1$ of $S^5$ as follows:
$$L_1(\cong\mathrm{T}^{\ast\perp}S^2)=\left\{
\left.(
\left[\begin{array}{c}
x_1\\
0\\
x_3\\
0\\
x_5\\
0
\end{array}\right],
\left[\begin{array}{c}
0\\
\xi_2\\
0\\
\xi_4\\
0\\
\xi_6
\end{array}\right]
)
\ \right|\ \|x\|=1, \xi_j\in\mathbb{R}(j=2,4,6)\right\},$$
$$L_2(\cong\mathrm{T}^{\ast\perp}S^1)=\left\{
\left.(
\left[\begin{array}{c}
x_1\\
0\\
x_3\\
0\\
0\\
0
\end{array}\right],
\left[\begin{array}{c}
0\\
\xi_2\\
0\\
\xi_4\\
\xi_5\\
\xi_6
\end{array}\right]
)
\ \right|\ \|x\|=1, \xi_j\in\mathbb{R}(j=2,4,5,6)\right\}.$$

Because these $S^2$ and $S^1$ are totally geodesic submanifolds of $S^5$, they are austere submanifolds. Hence their conormal bundles $\mathrm{T}^{\ast\perp}S^2$ and $\mathrm{T}^{\ast\perp}S^1$ are special Lagrangian submanifolds of $\mathrm{T}^\ast S^5$.
We use the polar coordinates $x_1=\cos\varphi_1\cos\varphi_2, x_3=\cos\varphi_1\sin\varphi_2, x_5=\sin\varphi_1$ for $L_1$ and $x_1=\cos\varphi, x_3=\sin\varphi$ for $L_2$.

Let $H$ be the $U(1)$-action of the Hopf-fibration $S^5\to\mathbb{C}P^2$, that is, the diagonal $U(1)\cong SO(2)$-action represented as follows:
$$H(\cong SO(2))=\left\{\left.\left[
\begin{array}{c|c|c}
h&&\\ \hline
&h&\\ \hline
&&h
\end{array}
\right]\in GL(6,\mathbb{R})
\ \right|\ 
h\in SO(2)\right\}.$$
The Lie algebra $\mathfrak{h}$ is given as follows:
$$\mathfrak{h}(\cong \mathfrak{so}(2))=\mathrm{span}\{\eta\},$$
here $\eta = \xi_{12}+\xi_{34}+\xi_{56}$ and $\xi_{ij}$ is as mentioned at the beginning of this section. Note that the isotropy subgroup of this $SO(2)$-action is trivial at any point $p\in L_1$ and $L_2$. Hence the condition (Imm-istp) holds for any point $p\in L_1$ and $L_2$.

We obtain an explicit expression of the moment map (\ref{na}) by direct calculations.
\begin{lem}\label{Claim.R}
Define $\mu_\eta$ by $\mu_\eta(z)=\langle\mu(z),\eta\rangle$ for $z\in\Phi(L_j)$ $(j=1,2)$. Then $\mu_\eta(z)$ equals to
$$
\left\{
\begin{array}{lc}
-\mathcal{K}(\|\xi\|)(\cos\varphi_1\cos\varphi_2\xi_2 + \cos\varphi_1\sin\varphi_2\xi_4+\sin\varphi_1\xi_6) &\mathrm{on}\, \Phi(L_1)\backslash\{\|\xi\|=0\},\\
-\mathcal{K}(\|\xi\|)(\cos\varphi\xi_2 + \sin\varphi\xi_4) &\mathrm{on}\, \Phi(L_2)\backslash\{\|\xi\|=0\}.
\end{array}
\right.
$$
Here,
$$\mathcal{K}(\|\xi\|)=\frac{u^\prime(\cosh(2\|\xi\|))\sinh(2\|\xi\|)}{\|\xi\|},$$
and $u$ is the function defined by $U(t) = u(\cosh t)$ for $U$ which is a solution of the ordinary differential equation (\ref{to}), that gives the K\"{a}hler potential of the Stenzel metric.
\end{lem}

Under these preparations, we obtain the following:
\begin{prop}\label{Claim.S}
Let $(M,I,\omega_\mathrm{Stz},\Omega_\mathrm{Stz},L_j,H)$ be as above. Let $V^{(j)}_c := L_j\cap\mu^{-1}(c)$ for each $c\in\mathfrak{h}^\ast$ and $j=1, 2$. 
\begin{enumerate}\setlength{\labelsep}{1em}\setlength{\itemindent}{1em}
\item[(1)]$H\cdot V^{(1)}_c$ is a special Lagrangian submanifold for any $c\in\mathfrak{h}^\ast$ such that $V^{(1)}_c\neq\emptyset$.
\item[(2)]$H\cdot V^{(2)}_c$ is a special Lagrangian submanifold for any $c\in\mathfrak{h}^\ast$ such that $V^{(2)}_c\neq\emptyset$.
\end{enumerate}
\end{prop}
\begin{proof}
Note that $Z(\mathfrak{h}^\ast) = \mathfrak{h}^\ast$ because $H$ is abelian. Thus (Istp-cnt) automatically holds. The condition (Imm-istp) also holds as mentioned above. Therefore we see that (Istp-cnt) and (Imm-istp) hold in the any case of (1) and (2).

(1) First we show the proposition above for $j=1$. This proof is based on Corollary \ref{Claim.Z}.
We will show in order (1-I) the perpendicular condition: the $H$-action satisfies (LagAng-$H$)$^\prime$ on $L_1$, and (1-II) the submanifold condition: $V^{(1)}_c\neq\emptyset$ is a submanifold of $M$ and (Lag-dim) holds for $(V^{(1)}_c, H, K)$.

(1-I) First we assume $\|\xi\|\neq 0$. By direct calculations, for $z\in\Phi(L_1)$, the fundamental vector $\eta_z^\#$ and $I_z\eta_z^\#$ are given as follows:
$$
\eta_z^\#
=
\cosh(\|\xi\|)
\left[\begin{array}{c}
0\\
\cos\varphi_1\cos\varphi_2\\
0\\
\cos\varphi_1\sin\varphi_2\\
0\\
\sin\varphi_1
\end{array}\right]
+
\sqrt{-1}\frac{\sinh(\|\xi\|)}{\|\xi\|}
\left[\begin{array}{c}
-\xi_2\\
0\\
-\xi_4\\
0\\
-\xi_6\\
0
\end{array}\right],
$$
$$I_z\eta^\#_z=\frac{\sinh(\|\xi\|)}{\|\xi\|}
\left[\begin{array}{c}
\xi_2\\
0\\
\xi_4\\
0\\
\xi_6\\
0
\end{array}\right]
+\sqrt{-1}\cosh(\|\xi\|)
\left[\begin{array}{c}
0\\
\cos\varphi_1\cos\varphi_2\\
0\\
\cos\varphi_1\sin\varphi_2\\
0\\
\sin\varphi_1
\end{array}\right].$$
On the other hand, using the coordinates above, we have a basis of $\mathrm{T}_z\Phi(L_1)$ as follows:
$$
\begin{aligned}
&\frac{\partial}{\partial\varphi_1}=\cosh(\|\xi\|)
(
-\sin\varphi_1\cos\varphi_2\mathbf{e}_1
-\sin\varphi_1\sin\varphi_2\mathbf{e}_3
+\cos\varphi_1\mathbf{e}_5
)
,\\
&\frac{\partial}{\partial\varphi_2}=\cosh(\|\xi\|)
(
-\cos\varphi_1\sin\varphi_2\mathbf{e}_1
+\cos\varphi_1\cos\varphi_2\mathbf{e}_3
)
,\\
&\frac{\partial}{\partial\xi_j}=\frac{\sinh(\|\xi\|)}{\|\xi\|}\xi_jx+\sqrt{-1}\left\{\frac{\xi_j}{\|\xi\|^2}\mathcal{F}\xi+\frac{\sinh(\|\xi\|)}{\|\xi\|}\mathbf{e}_j\right\}\quad  (j=2,4,6),\\
\end{aligned}
$$
where
$$
\mathcal{F}=\mathcal{F}(\|\xi\|)=\cosh(\|\xi\|)-\frac{\sinh(\|\xi\|)}{\|\xi\|}.
$$
Since $L_1$ is a Lagrangian submanifold of a K\"{a}hler manifold $M$, it is sufficient for verifying $\eta^\#_z\in\mathrm{T}_z^\perp \Phi(L_1)$ to show $I_z\eta^\#_z \in \mathrm{T}_z\Phi(L_1)$. For generating the imaginary part of $I_z\eta^\#_z$ by $\left(\frac{\partial}{\partial \xi_j}\right)\,(j=2,4,6)$, the following is necessary: there exists $(a_2, a_4, a_6)\in\mathbb{R}^3\backslash\{\mathbf{0}\}$ which satisfies
\begin{equation}\label{eqa}
A
\left[
\begin{array}{c}
a_2\\
a_4\\
a_6
\end{array}
\right]
=
\left[
\begin{array}{c}
\cosh(\|\xi\|)\cos\varphi_1\cos\varphi_2\\
\cosh(\|\xi\|)\cos\varphi_1\sin\varphi_2\\
\cosh(\|\xi\|)\sin\varphi_1
\end{array}
\right],
\end{equation}
where
$$
A
=
\left[
\begin{array}{ccc}
\frac{\xi_2^2}{\|\xi\|^2}\mathcal{F}+\frac{\sinh(\|\xi\|)}{\|\xi\|}
&\frac{\xi_2\xi_4}{\|\xi\|^2}\mathcal{F}
&\frac{\xi_2\xi_6}{\|\xi\|^2}\mathcal{F}\\
\frac{\xi_2\xi_4}{\|\xi\|^2}\mathcal{F}
&\frac{\xi_4^2}{\|\xi\|^2}\mathcal{F}+\frac{\sinh(\|\xi\|)}{\|\xi\|}
&\frac{\xi_4\xi_6}{\|\xi\|^2}\mathcal{F}\\
\frac{\xi_2\xi_6}{\|\xi\|^2}\mathcal{F}
&\frac{\xi_4\xi_6}{\|\xi\|^2}\mathcal{F}
&\frac{\xi_6^2}{\|\xi\|^2}\mathcal{F}+\frac{\sinh(\|\xi\|)}{\|\xi\|}
\end{array}
\right].
$$
With the use of series expansion of hyperbolic functions, we can see rank$A=3$ if $\|\xi\|\neq 0$. Hence (\ref{eqa}) has a non-trivial solution for each $\|\xi\|\neq 0$. For this solution $(a_2, a_4, a_6)$, we can verify that there exists $(b_1, b_2)\in\mathbb{R}^2$ which satisfies
$$
b_1\frac{\partial}{\partial\varphi_1}+b_2\frac{\partial}{\partial\varphi_2}+a_2\frac{\partial}{\partial\xi_2}+a_4\frac{\partial}{\partial\xi_4}+a_6\frac{\partial}{\partial\xi_6}
=I_z\eta^\#_z
$$
by using the relation (\ref{eqa}).
Thus if $\|\xi\|\neq 0$, (LagAng-$H$)$^\prime$ holds.

When $\|\xi\|=0$, by taking a limit $\|\xi\|\to 0$, we have
$$\frac{\partial}{\partial\xi_j}\to\sqrt{-1}\mathbf{e}_j\quad (j=2,4,6).$$
Thus we can also verify that $I_z\eta^\#_z\in\mathrm{T}_z\Phi(L_1)$ when $\|\xi\|=0$.

(1-II) Note that $\mu(L_1\cap\{\|\xi\|=0\})=0$. When $\|\xi\|\neq 0$, we use the following fact: There exists a neighborhood $U_p$ around $p\in V^{(1)}_c$ in $L_1$ such that $V^{(1)}_c\cap U_p$ is a submanifold of $L_1$ (therefore of $M$), if $(\nabla \mu_\eta)_p\neq\mathbf{0}\in\mathrm{T}_pL_1$, and then $\dim V^{(1)}_c = \dim L_1- 1=5-1=4$. Here $\nabla$ is the gradient with respect to the induced metric $\iota^\ast g_{\mathrm{Stz}}$ by the inclusion map $\iota:L_1\to M$. Using the properties of Stenzel's K\"{a}hler potential $u$, we can verify that $\nabla\mu_\eta\neq\mathbf{0}$ on $L_1\backslash\{\|\xi\|=0\}$ by direct calculations. When $\|\xi\|=0$, it is sufficient to verify $V^{(1)}_0$ is a submanifold of $M$. By the expression of the moment map in Lemma \ref{Claim.R}, we have
$$
V^{(1)}_0=\left\{\left.(
\left[
\begin{array}{c}
x_1\\
0\\
x_3\\
0\\
x_5\\
0
\end{array}
\right],
\left[
\begin{array}{c}
0\\
\xi_2\\
0\\
\xi_4\\
0\\
\xi_6
\end{array}
\right]
)
\ \right|\ 
\|x\|=1, \left[
\begin{array}{c}
x_1\\
x_3\\
x_5
\end{array}
\right]
\centerdot
\left[
\begin{array}{c}
\xi_2\\
\xi_4\\
\xi_6
\end{array}
\right]=0\right\}.
$$
This is diffeomorphic to $\mathrm{T}S^2$. Therefore $V^{(1)}_c\neq\emptyset$ is a submanifold of $M$ for any $c\in\mathfrak{h}^\ast$, and (Lag-dim) holds for $(V^{(1)}_c, H, K)$. Thus we have proven (1) of the proposition.

(2) This proof is based on Corollary \ref{Claim.M}.
(2-I) the generalized perpendicular condition: the $H$-action satisfies (LagAng-$H$) on $L_2$. To show it, first we assume $\|\xi\|\neq 0$. By direct calculations, we have
$$
\eta_z^\#
=
\cosh(\|\xi\|)
\left[
\begin{array}{c}
0\\
\cos\varphi\\
0\\
\sin\varphi\\
0\\
0
\end{array}
\right]
+
\sqrt{-1}\frac{\sinh(\|\xi\|)}{\|\xi\|}
\left[
\begin{array}{c}
-\xi_2\\
0\\
-\xi_4\\
0\\
-\xi_6\\
\xi_5
\end{array}
\right]
$$
We set the following strategy. First we decompose $I_z\eta^\#_z$ as follows:
$$I_z\eta^\#_z=I_z(\eta^\#_z)_1+I_z(\eta^\#_z)_2,$$
here
$$I_z(\eta^\#_z)_1=\frac{\sinh(\|\xi\|)}{\|\xi\|}
\begin{bmatrix}
\xi_2\\
0\\
\xi_4\\
0\\
0\\
0
\end{bmatrix}
+\sqrt{-1}\cosh(\|\xi\|)
\begin{bmatrix}
0\\
\cos\varphi\\
0\\
\sin\varphi\\
0\\
0
\end{bmatrix}
,
\ I_z(\eta^\#_z)_2=\frac{\sinh(\|\xi\|)}{\|\xi\|}
\begin{bmatrix}
0\\
0\\
0\\
0\\
\xi_6\\
-\xi_5
\end{bmatrix}.$$
Then assume that $I_z(\eta^\#_z)_1\in\mathrm{T}_z\Phi(L_2)$. Since $I_z(\eta^\#_z)_2$ clearly has no $\mathrm{T}_z\Phi(L_2)$-components, we see that the decomposition above is a direct decomposition with respect to $\mathrm{T}_zQ^n=\mathrm{T}_z\Phi(L_2)\oplus\mathrm{T}_z^\perp\Phi(L_2)$. Since
$$
(\eta^\#_z)_2=\sqrt{-1}\frac{\sinh(\|\xi\|)}{\|\xi\|}\left[
\begin{array}{c}
0\\
0\\
0\\
0\\
-\xi_6\\
\xi_5
\end{array}
\right]
$$
and $\mu$ depend neither on fifth nor sixth component of the imaginary part of $\mathbb{C}^{n+1}\cong \mathrm{T}_z\mathbb{C}^{n+1}\supset\mathrm{T}_zQ^n$, we have
$$\langle(d\mu_\eta)_z,(\eta^\#_z)_2\rangle=0,$$
namely $(\eta^\#_z)_2\in\mathrm{T}_z\mu^{-1}(\mu(z))$.
Hence we have that $(\eta^\#_z)_2\in\mathrm{T}_z\mu^{-1}(\mu(z))\cap\mathrm{T}_z\Phi(L_2)=\mathrm{T}_z\Phi(V^{(2)}_{\mu(z)})$.
Noting that $(\eta^\#_z)_1\neq\mathbf{0}$ for any $z\in\Phi(L_2)$, we thus see that (LagAng-$H$) holds if $I_z(\eta^\#_z)_1\in\mathrm{T}_z\Phi(L_2)$.
We can verify $I_z(\eta^\#_z)_1\in\mathrm{T}_z\Phi(L_2)$ actually as in the case of (1). Hence (LagAng-$H$) holds at any point $p\in L_2\backslash\{\|\xi\|=0\}$. When $\|\xi\| =0$, we can also compute $I_z\eta^\#_z\in\mathrm{T}_z\Phi(L_2)$ by taking a limit $\|\xi\|\to 0$. This indicates that (LagAng-$H$)$^\prime$ holds at any point $p\in L_2\cap\{\|\xi\| =0\}$ more strictly. Thus we see that (2-I) holds.

(2-II) the submanifold condition: $V^{(2)}_c\neq\emptyset$ is a submanifold of $M$ and (Lag-dim) holds for $(V^{(2)}_c, H, K)$. For $c\neq 0$, we can verify this as in the case of (1) by considering the gradient $\nabla \mu_\eta$. For $c=0$, $V^{(2)}_0$ is expressed as follows: 
$$
V^{(2)}_0=\left\{\left.(
\left[
\begin{array}{c}
x_1\\
0\\
x_3\\
0\\
0\\
0
\end{array}
\right],
\left[
\begin{array}{c}
0\\
\xi_2\\
0\\
\xi_4\\
\xi_5\\
\xi_6
\end{array}
\right]
)
\ \right|\ 
\|x\|=1, \left[
\begin{array}{c}
x_1\\
x_3
\end{array}
\right]
\centerdot
\left[
\begin{array}{c}
\xi_2\\
\xi_4
\end{array}
\right]=0\right\}.
$$
Ignoring fifth and sixth components, this is diffeomorphic to $TS^1\cong S^1\times\mathbb{R}$. Fifth and sixth components constitute a plane unrelated to the base manifold. Thus we see $V^{(2)}_0\cong S^1\times\mathbb{R}^3$ and $(V^{(2)}_0, H, K)$ satisfies (Lag-dim). Thus we see that (2-II) holds.
\end{proof}
\subsection{The case of $H=SO(2)\times SO(2)\times SO(3), L=\mathrm{T}^{\ast\perp}S^2\subset\mathrm{T}^\ast S^6$}\label{Section.E.2}

Let $M$ be the cotangent bundle of $6$-sphere $\mathrm{T}^\ast S^6$ and $L$ the conormal bundle of a totally geodesic submanifold $S^2$ of $S^6$ as follows:
$$L(= \mathrm{T}^{\ast\perp}S^2)=\left\{\left.
(
\left[
\begin{array}{c}
x_1\\
0\\
x_3\\
0\\
x_5\\
0\\
0
\end{array}
\right],
\left[
\begin{array}{c}
0\\
\xi_2\\
0\\
\xi_4\\
0\\
\xi_6\\
\xi_7
\end{array}
\right]
)
\ \right|\ \|x\|=1, \xi_j\in\mathbb{R} (j=2,4,6,7)\right\}.$$

Define $H$ as follows:
$$
\begin{aligned}
&H (\cong SO(2)\times SO(2)\times SO(3))\\
=&\left\{\left.
\left[
\begin{array}{c|c|c}
h_1&0&0\\ \hline
0&h_2&0\\ \hline
0&0&h_3
\end{array}
\right]\in GL(7,\mathbb{R})
\ \right|\ 
h_1, h_2\in SO(2), h_3\in SO(3)
\right\}.
\end{aligned}
$$
Note that $H$ is non-abelian.
The Lie algebra $\mathfrak{h}$ of $H$ is given as follows:
$$
\mathfrak{h}(\cong\mathfrak{so}(2)\oplus\mathfrak{so}(2)\oplus\mathfrak{so}(3))=\mathrm{span}\{\xi_{12},\xi_{34},\xi_{56},\xi_{57},\xi_{67}\},
$$
here $\xi_{ij}$ is as mentioned at the beginning of this section.

We obtain an explicit expression of the moment map of (\ref{na}) by direct calculations. Define $\mu_{ij}$ for the basis $(\xi_{12},\xi_{34},\xi_{56},\xi_{57},\xi_{67})$ of $\mathfrak{h}$ and $z\in\Phi(L)$ by $\mu_{ij}(z)=\langle \mu(z),\xi_{ij}\rangle$.

\begin{lem}\label{Claim.P}
For $(M, I, \omega, H, \mu)$ above, and $z\in\Phi(L)\backslash\{\|\xi\|=0\}$, we have
$$
\begin{aligned}
&\mu_{12}(z)=-\mathcal{K}(\|\xi\|)\cos\varphi_1\cos\varphi_2\xi_2,\\
&\mu_{34}(z)=-\mathcal{K}(\|\xi\|)\cos\varphi_1\sin\varphi_2\xi_4,\\
&\mu_{56}(z)=-\mathcal{K}(\|\xi\|)\sin\varphi_1\xi_6,\\
&\mu_{57}(z)=-\mathcal{K}(\|\xi\|)\sin\varphi_1\xi_7,\\
&\mu_{67}(z)\equiv 0.
\end{aligned}
$$
Here, we use the polar coordinates $x_1=\cos\varphi_1\cos\varphi_2, x_3=\cos\varphi_1\sin\varphi_2, x_5=\sin\varphi_1$, and
$$\mathcal{K}(\|\xi\|)=\frac{u^\prime(\cosh(2\|\xi\|))\sinh(2\|\xi\|)}{\|\xi\|}.$$
\end{lem}

We obtain the following result:
\begin{prop}\label{Claim.Q}
Let $(M, I, \omega_{\mathrm{Stz}}, \Omega_{\mathrm{Stz}},L, H)$ be as above. Define a rank two subbundle $\hat{L}$ of $L$ as follows:
$$\hat{L}
=
\left\{\left.
(
\left[
\begin{array}{c}
x_1\\
0\\
x_3\\
0\\
x_5\\
0\\
0
\end{array}
\right],
\left[
\begin{array}{c}
0\\
\xi_2\\
0\\
\xi_4\\
0\\
0\\
0
\end{array}
\right]
)
\ \right|\ 
\|x\|=1, \xi_j\in\mathbb{R}\ (j=2,4)\right\}.
$$
Let $\hat{L}^\mathrm{pr}$ be the set of all points $p\in\hat{L}$ such that the isotropy subgroup $H_p$ satisfies $H_p\subset H_q$ for all $q\in\hat{L}$.
For $(c_1,c_2)\in \mathbb{R}^2$, define $V_{(c_1,c_2)}$ and $\hat{V}_{(c_1,c_2)}$ by 
$$
\begin{aligned}
V_{(c_1,c_2)}&=\hat{L}^{\mathrm{pr}}\cap\{p\in M\mid\mu_{12}(p)=c_1, \mu_{34}(p)=c_2, \mu_{ij}(p)=0\},\\
\hat{V}_{(c_1,c_2)}&=\hat{L}\cap\{p\in M\mid\mu_{12}(p)=c_1, \mu_{34}(p)=c_2, \mu_{ij}(p)=0\ \},
\end{aligned}
$$
here $(i,j)=(5,6),(5,7),(6,7)$.
Then for any $(c_1,c_2)\neq(0,0)\in\mathbb{R}^2$ such that $V_{(c_1,c_2)}\neq\emptyset$, $H\cdot V_{(c_1,c_2)}$ is a special Lagrangian submanifold of $M$, and $H\cdot \hat{V}_{(0,0)}$ is a union of five connected special Lagrangian submanifolds of $M$.
\end{prop}
\begin{proof}
The proof for $V_{(c_1,c_2)}$ is based on Corollary \ref{Claim.Z}, and one for $\hat{V}_{(0,0)}$ on direct calculations. As we saw in Remark \ref{Claim.O}, $V_{(c_1,c_2)}$ has to be included in the inverse image of the center $Z(\mathfrak{h}^\ast)$ of $\mathfrak{h}^\ast$ with the moment map $\mu$.
Hence, noting the $\mathfrak{so}(3)$-part of $\mathfrak{h}\cong\mathfrak{so}(2)\oplus\mathfrak{so}(2)\oplus\mathfrak{so}(3)$, we see that we can apply our construction for the part such that $\mu_{ij}(p)=0\ ((i,j)=(5,6),(5,7),(6,7))$ in $L$. This indicates that $\xi_6=\xi_7=0$ is necessary. That is, the place in where we have to check the conditions of Corollary \ref{Claim.N} is $\hat{L}\subset L$.

By the definition of $V_{(c_1,c_2)}$, at any point $p\in V_{(c_1,c_2)}$, the isotropy subgroup $H_p$ is the following one-parameter subgroup $K$ generated by $\xi_{67}$: 
$$
K(\cong SO(2))=\left\{\left.\left[
\begin{array}{c|c}
E_5&\\ \hline
&h\\
\end{array}
\right]
\ \right|\ 
h\in SO(2)\right\},
$$
here $E_5$ is the unit $5\times 5$-matrix. Thus we see that (Imm-istp) holds.

First we prove the proposition for $V_{(c_1,c_2)}$ for $(c_1,c_2)\neq (0,0)$. Since $\mu(\hat{L}\cap\{\|\xi\|=0\})=\mathbf{0}$, we can assume $\|\xi\|\neq 0$ in this case.
As same as Proposition \ref{Claim.S}, conditions we have to check are the followings:  (I) the perpendicular condition: the $H$-action satisfies (LagAng-$H$)$^\prime$ on $\hat{L}\backslash\{\|\xi\|=0\}$, and (II) the submanifold condition: $V_{(c_1,c_2)}\neq\emptyset$ is a submanifold of $M$ and (Lag-dim) holds for $(V_{(c_1,c_2)}, H, K)$. We can verify these in the same way as Proposition \ref{Claim.S}.

Finally we study $\hat{V}_{(0,0)}$ generally rather than $V_{(0,0)}$, including non-principal points. By Lemma \ref{Claim.P}, We obtain that
$$\hat{V}_{(0,0)}
=
\left\{(\left.\left[
\begin{array}{c}
x_1\\
0\\
x_3\\
0\\
x_5\\
0\\
0\\
\end{array}
\right],
\left[
\begin{array}{c}
0\\
\xi_2\\
0\\
\xi_4\\
0\\
0\\
0
\end{array}
\right])
\ \right|\ 
\begin{matrix}
\|x\|=1,\\
\xi_2,\xi_4\in\mathbb{R},\\
x_1\xi_2=x_3\xi_4=0
\end{matrix}
\right\}.
$$
$\hat{V}_{(0,0)}$ is not a smooth manifold. However it is a union, which is not disjoint, of the following five connected manifolds:
$$
\hat{V}_{(0,0)}
=
\hat{V}_{(0,0)}^{S^2}
\cup
\hat{V}_{(0,0), (1)}^{S^1\times\mathbb{R}}
\cup
\hat{V}_{(0,0), (3)}^{S^1\times\mathbb{R}}
\cup
\hat{V}_{(0,0), (1)}^{\mathbb{R}^2}
\cup
\hat{V}_{(0,0), (-1)}^{\mathbb{R}^2},
$$
here
$$
\hat{V}_{(0,0)}^{S^2}=\left\{(\left.
\begin{bmatrix}
x_1\\
0\\
x_3\\
0\\
x_5\\
0\\
0\\
\end{bmatrix}
,\mathbf{0})\ \right|\ \|x\|=1\right\},\quad
\hat{V}_{(0,0),(1)}^{S^1\times\mathbb{R}}=\left\{\left.(
\begin{bmatrix}
0\\
0\\
x_3\\
0\\
x_5\\
0\\
0\\
\end{bmatrix}
,
\begin{bmatrix}
0\\
\xi_2\\
0\\
0\\
0\\
0\\
0
\end{bmatrix}
)
\,\right|
\begin{matrix}
\|x\|=1,\\
\xi_2\in\mathbb{R}
\end{matrix}
\right\},$$
$$
\hat{V}_{(0,0),(3)}^{S^1\times\mathbb{R}}
=
\left\{(\left.
\begin{bmatrix}
x_1\\
0\\
0\\
0\\
x_5\\
0\\
0\\
\end{bmatrix}
,
\begin{bmatrix}
0\\
0\\
0\\
\xi_4\\
0\\
0\\
0
\end{bmatrix}
)
\,\right|
\begin{matrix}
\|x\|=1,\\
\xi_4\in\mathbb{R}
\end{matrix}
\right\},\quad
\hat{V}_{(0,0), (\epsilon)}^{\mathbb{R}^2}
=
\left\{(\left.
\begin{bmatrix}
0\\
0\\
0\\
0\\
\epsilon\\
0\\
0\\
\end{bmatrix}
,
\begin{bmatrix}
0\\
\xi_2\\
0\\
\xi_4\\
0\\
0\\
0
\end{bmatrix}
)
\ \right|\ 
\xi_2,\xi_4\in\mathbb{R}\right\},
$$
and $\epsilon =\pm 1$.
We can see that each set $\hat{V}_{(0,0)}^W$ is a $2$-dimensional connected submanifold of $M$ diffeomorphic to $W$. Each $\hat{V}_{(0,0)}^W$ has non-principal orbits with respect to the action of $H$ on $H\cdot \hat{V}_{(0,0)}^W$. Hence it does not satisfy (Imm-istp). However we can directly verify that each $H\cdot\hat{V}_{(0,0)}^W$ for $\hat{V}_{(0,0)}^{S^2}$, $\hat{V}_{(0,0),(j)}^{S^1\times\mathbb{R}} (j=1,3)$, and $\hat{V}_{(0,0), (\epsilon)}^{\mathbb{R}^2} (\epsilon=\pm 1)$ is a special Lagrangian submanifold of $M$ diffeomorphic to $S^6$, $\mathrm{T}^{\ast\perp}S^4$, and $\mathrm{T}^{\ast\perp}S^2$ respectively.
\end{proof}

We chose $SO(2)\times SO(2)\times SO(3)$ as a Lie group $H$ for the special Lagrangian submanifold $L\subset \mathrm{T}^\ast S^6$ of Proposition \ref{Claim.Q} rather than $SO(2)\times SO(5)$ because of two reasons. First, since the center of the Lie algebra $\mathfrak{h}\cong\mathfrak{so}(2)\oplus\mathfrak{so}(2)\oplus\mathfrak{so}(3)$ has two dimensions, we could obtain two-parameters of special Lagrangian submanifolds $H\cdot V_{(c_1,c_2)}$. Second, for $p, q\in\mathbb{N} $ such that $p+q=n+1$, special Lagrangian submanifolds which are $SO(p)\times SO(q)$-invariant in $(\mathrm{T}^\ast S^n, I, g_{\mathrm{Stz}}, \omega_{\mathrm{Stz}})$ have been already obtained by Hashimoto and Sakai in \cite{Hashimoto and Sakai}, and they showed that such special Lagrangian submanifolds are cohomogeneity one with respect to $SO(p)\times SO(q)$.  In the case of Proposition \ref{Claim.Q}, we can verify that $SO(2)\times SO(2)\times SO(3)$ acts on $H\cdot V_{(c_1,c_2)}$ with cohomogeneity two.



\end{document}